\documentclass[reqno,10pt]{amsart}
\usepackage[dvipsnames,usenames]{color} 
\usepackage[toc,page]{appendix}

\usepackage{mathtools}
\usepackage{tikz}  
\usepackage{circledsteps}

\usepackage{hyperref}
\usepackage[mathscr]{euscript} 

\usepackage{cite}
\usepackage{graphicx}
\usepackage[all]{xy} \xyoption{arc} \xyoption{color}

\usepackage{amsmath} 
\usepackage{amsthm} 
\usepackage{amsfonts}  
\usepackage{amssymb} 

\usepackage{verbatim}  
\usepackage{enumerate}

\usepackage{mathrsfs}

\def\XXint#1#2#3{{\setbox0=\hbox{$#1{#2#3}{\int}$}
		\vcenter{\hbox{$#2#3$}}\kern-.5\wd0}}

\newcommand{\RR}{\mathbb R}
\newcommand{\Rd}{{\mathbb R}^3}
\newcommand{\Cd}{{\mathbb C}^3}

\newcommand{\NN}{\mathbb N}
\newcommand{\id}{\mathbf 1}

\newcommand{\Zd}{{\mathbb Z^3}}

\newcommand{\Td}{{\mathbb T^3}}
\newcommand{\TT}{{\mathbb T^3}}

\newcommand{\divv}{{\rm div \,}}

\newcommand{\spt}{\operatorname{spt}}
\newcommand{\rot}{\mathcal{R}}

\newcommand{\vertiii}[1]{{\left\vert\kern-0.25ex\left\vert\kern-0.25ex\left\vert #1 
		\right\vert\kern-0.25ex\right\vert\kern-0.25ex\right\vert}}

\newcommand{\qm}{{q\bmod3}}
\newcommand{\piloc}{\Pi_q^{\mathrm{local}}}
\newcommand{\pinonloc}{\Pi_q^{\mathrm{non-local}}}

\newcounter{comentcount}
\newcounter{teocount}
\newtheorem{lem}{Lemma}
\newtheorem{prop}{Proposition}

\newtheorem{teo}[teocount]{Theorem}

\newtheorem{conj}{Conjecture}

\title[]{Anomalous energy flux in critical $L^p$-based spaces}

\date{\today}
\author[J. Burczak]{Jan Burczak}
\address{Mathematisches Institut, Universit\"at Leipzig, D-04109 Leipzig, Germany}
\email{burczak@math.uni-leipzig.de}
\author[G. Sattig]{Gabriel Sattig}
\address{Mathematisches Institut, Universit\"at Leipzig, D-04109 Leipzig, Germany}
\email{sattig@math.uni-leipzig.de}

\thanks{This research was supported by the ERC Grant Agreement No. 724298. The authors thank S. Modena and L. Sz\'ekelyhidi for useful discussions.}

\begin{document}

    \begin{abstract}
		We construct a three-dimensional vector field that exhibits positive energy flux at every Littlewood-Paley shell and has the best possible regularity in $L^p$-based spaces, $p \le 3$; in particular, it belongs to $H^{(\frac56)^-}$.
		\end{abstract}
	
    

	\maketitle

	\section{Introduction}
	\label{sec:intro}
	\subsection{Anomalous dissipation: What is known}
	
	In his famous 1949 work \cite{Ons49} Onsager conjectured energy dissipation `without final assistance of viscosity' in fluids. In today's mathematical literature such phenomenon is referred to as `anomalous dissipation' of solutions to the Euler equation. Onsager's original conjecture was stated in terms of Hölder spaces and is by now fully resolved: Constantin, E and Titi \cite{ConETit94} proved conservation of energy of any weak Euler solution which is $C^{1/3+\varepsilon}$-regular in space,
	while Isett \cite{Ise18} constructed non-conservative $C^{1/3-\varepsilon}$ Euler solutions, using the convex integration technique developed by De~Lellis and Sz\'ekelyhidi \cite{DLScontinuous}. For the truly dissipative case, see Buckmaster, De~Lellis, Sz\'{e}kelyhidi, and Vicol \cite{BDLSV19}.
	
	Both results remain valid if stated in terms of $L^3$-based Sobolev or Besov spaces, i.e.\ for $W^{1/3 \pm \varepsilon,3}$ or $B_{3,\infty}^{1/3\pm\varepsilon}$: the dissipative one trivially, the conservative part was shown by Duchon and Robert \cite{Duchon-Robert}. The exponent~$3$ appears naturally via the cubic term in the energy balance equation, but less is known if the integrability is below $3$, in particular for $H^s=W^{s,2}$. The best known result on the conservative side is the trivial one: by Sobolev embedding $H^{5/6+\varepsilon} \hookrightarrow W^{1/3+\varepsilon,3}$.
On the non-conservative side, recently $H^{1/2-\varepsilon}$ has been reached in Buckmaster, Masmoudi, Novack, and Vicol \cite{BMNV22} (which was later extended by Novack and Vicol \cite{NV22}).
While the exponent $1/2$ may be critical
for full ($h$-principle-type) flexibility of Euler solutions, 
it is conjectured (cf.~Problem~5 in the survey \cite{BuckmasterVicol-survey} by Buckmaster and Vicol) that $5/6$ is the threshold for energy conservation.
	\begin{conj}
		For any $s<\frac{5}{6}$ there are weak solutions to the Euler equation in $C_tH^s_x$ with strictly decreasing kinetic energy.
	\end{conj}
	In this work we do not prove the conjecture, but provide evidence in its favour.
	
	\subsection{Energy flux and statement of the main result}
	The first step for proving conservation of energy is usually to regularize the Euler equation and to test it with the regularised solution:
	\begin{equation}\label{eq:mot1}
	\frac{d}{dt} \frac{1}{2} \int \left(S_q u\right)^2 \,dx= \int S_q u \cdot \partial_t S_q u \,dx = - \int S_q u \cdot S_q \left( u\cdot \nabla u\right) \,dx. \end{equation}
	Here $S_q$ denotes any linear regularizing operator such that $S_q\to Id$ as $q\to\infty$; in our case it will be the Littlewood-Paley projection onto functions with frequencies $\lesssim 2^q$. 
	Note that the pressure term vanishes due to incompressibility and linearity of $S_q$.

   The second step is showing that the right hand side of \eqref{eq:mot1}, called the energy flux (towards high frequencies), vanishes as $q\to\infty$, which holds if $u$ is a divergence free field with enough regularity; this step may be thus completely detached from the fact that $u$ solves the Euler equation!

	By \eqref{eq:mot1}, vanishing of the energy flux is a sufficient condition for conservation of energy. Even though it may not be a necessary condition, non-vanishing of the flux suggests an energy cascade as described by Richardson (cf.\ Frisch \cite{Frisch_book}), i.e.~continuous transport of energy from low to high and higher frequency structures, which is in turn seen as the mechanism of anomalous dissipation. Therefore, constructing a non-vanishing energy flux of certain regularity is a strong premise for non-conservative Euler solutions with that regularity. Such flux construction, with certain further desired properties, is the subject of this note.
	
	Let us define for a distributionally divergence-free vector field $u: \Td\to\Rd$
	\[
	\Pi_q (u)= \int_{\TT} S_q (u \otimes u) : \nabla S_q (u) \, dx,
	\]
	where $S_q$ denotes a Littlewood-Paley projection, to be precised later. The right hand side of \eqref{eq:mot1}, with $S_q$ disambiguated as a Littlewood-Paley projection, after applying incompressibility and integration by parts, is precisely our $\Pi_q (u)$.
	With this definition, and the standard notion of Besov spaces $B^{s}_{p, r}$ (defined in Section~\ref{sec:construction}) we state
	\begin{teo}\label{thm}
		For any $c \in \RR$ and $\delta>0$, there exists a divergence-free vector field $U \in B^{5/6}_{2,\infty} (\TT)$ such that
		\begin{equation}
		\liminf_{q \to \infty} \Pi_q (U) >c-\delta \qquad \text{ and } \qquad \limsup_{q\to\infty} \Pi_q(U)<c+\delta. 
		\label{limsup-liminf}
		\end{equation}
		In fact, $U \in {{B}^{\frac{3}{p}-\frac{2}{3}}_{p,\infty}} (\TT)$ with any $p \in (1, \frac92)$ and $U \in {\dot {B}^{\frac{3}{p}-\frac{2}{3}}_{p,\infty}} (\TT)$ with any $p \in (1, \infty]$.
	\end{teo}

	\subsection{Remarks and comparison to previous results}
	\begin{enumerate}[(i)]
			\item The theorem is sharp concerning regularity of $U$. Indeed, for any $p<\infty$ the flux $\Pi_q(U)$ vanishes for $q\to \infty$ for every $U\in B_{2,r}^{5/6}$. This follows from embedding into $B_{3,r}^{1/3}$ and Theorem 3.3 of Cheskidov, Constantin, Friedlander, and Shvydkoy \cite{CCFS}.
			\item A proper limit seems possible with a more careful choice of parameters. We did not pursue this here, as the main objective was obtaining a positive $\liminf$ for critical spaces in three space dimensions.
		\item Our construction is strongly inspired by Cheskidov, Filho, Lopes, and Shvydkoy  \cite{CNLS16}, which in turn develops ideas of Eyink \cite{Eyink}. Regularity-wise, a two dimensional sharp result is sketched in \cite{CNLS16}. The vector field $\mathcal{U}$ described there enjoys $\limsup_{q \to \infty} \Pi_q(\mathcal{U})\neq 0$, while $\liminf_{q \to \infty} \Pi_q(\mathcal{U})=0$. In fact, \cite{CNLS16} points at Cheskidov and Shvydkoy \cite{CheShv10} for details of the construction  aimed at sharp regularity for $p<3$. A straightforward attempt to fill in those details may result in non-zero values appearing extremely sparsely in the sequence $\Pi_q(\mathcal{U})$, which means that 'typically' $\Pi_q(\mathcal{U})=0$. This
        is unsatisfactory from the perspective of energy cascade. Indeed, any $\Pi_q(\mathcal{U})\approx 0$ disrupts the transport of energy, which is predominantly local (from lower to slightly higher frequencies; this locality remains in accordance with turbulence literature). Hence there cannot be a continuous cascade related to $\mathcal{U}$. 
		\item Another approach towards anomalous dissipation is proposed by Cheskidov and Luo \cite{ChLuo21}: the authors construct almost fully intermittent vector fields in 3 or more dimensions, which are not stationary but instead use time as another degree of freedom. More precisely, the construction features a `time-delayed' energy cascade and the definition of energy flux involves averaging in time. This is a very interesting approach, especially towards possible blow-up scenarios, but it does not satisfy the energy cascade heuristics sketched above.
	\end{enumerate}
	
	\subsection{Organisation} In Section~\ref{sec:construction} we detail the construction of the vector field $U$ and gather its properties. In Section~\ref{sec:decomposition} we decompose the energy flux of $U$, which facilitates computations of the final Section~\ref{sec:fin}, where the proof of Theorem \ref{thm} is concluded.
	
	\section*{Acknowledgements}
	This research was supported by the ERC Grant Agreement No. 724298.
	The authors thank S. Modena and L. Sz\'ekelyhidi for useful discussions.

	\section{The construction}
	\label{sec:construction}
	Take a smooth function
	\[
	\varphi:[0,\infty) \to [0,1] ,\quad \varphi (\xi) = \begin{cases} &1 \quad  \xi \le \frac{\sqrt{5}}{2}+2\epsilon\\
	&0 \quad \xi \ge 2-4\epsilon,
	\end{cases}
	\]
	with $\epsilon$ small enough, to be fixed later.
	Let $\psi(\xi) = \varphi(2 \xi) - \varphi(\xi)$. The Littlewood-Paley low-frequency and dyadic projections are, respectively, $\widehat{S_q u}(k)\coloneqq \varphi (\lambda_{q}^{-1}|k|) \hat u (k)$ and $\widehat{\Delta_q u}(k)\coloneqq \psi (\lambda_{q}^{-1}|k|) \hat u (k)$, where $\lambda_{q}=2^q$. 
	
	\subsection{The planar construction}\label{ssec:planar}
In this section we construct the planar vector field $U_0$ along \cite{CNLS16}.  Fix the plane $P_0\coloneqq\{(x_1,x_2,x_3) \in \RR^3 | \; x_3=0\}$. Define the frequency vectors
	\[ F_0^{\Circled{1}} \coloneqq (0,1,0) ,\ F_0^{\Circled{2}} \coloneqq (2,0,0) ,\ F_0^{\Circled{3}} \coloneqq (2,1,0) \]
	and amplitude vectors
	\[  V_0^{\Circled{1}} \coloneqq (1,0,0) , \  V_0^{\Circled{2}} \coloneqq (0,1,0) , \  V_0^{\Circled{3}} \coloneqq (-1,2,0). \]
	The planar 'skeleton field' $S$ is the sum over $q \in 3 \NN$ of $s_q= s_q^{\Circled{1}}+s_q^{\Circled{2}}+s_q^{\Circled{3}}$, where
	\[
	\begin{aligned}
	s_q^{\Circled{1}} &\coloneqq \lambda_q^{-\frac{1}{3}} V_0^{\Circled{1}} \sin \left(\lambda_q F_0^{\Circled{1}} \cdot x\right), \\
	s_q^{\Circled{2}} &\coloneqq \lambda_q^{-\frac{1}{3}} V_0^{\Circled{2}}  \cos \left( \lambda_q F_0^{\Circled{2}} \cdot x\right), \\
	s_q^{\Circled{3}} &\coloneqq \lambda_q^{-\frac{1}{3}} V_0^{\Circled{3}}  \cos \left( \lambda_q F_0^{\Circled{3}} \cdot x\right) 
	\end{aligned}
	\]
	The skeleton field $S$ suffices to obtain the $\limsup$ result in $B^{\frac13}_{3,\infty} (\TT)$; in particular $\divv S = 0$ thanks to $ F_0^{\Circled{i}} \cdot V_0^{\Circled{i}} =0$.

	Since the Bernstein inequality implies the Sobolev inequality:
\[
	\| \Delta_q f\|_{L^3(\TT)} \le C \lambda^\frac12_q \| \Delta_q f\|_{L^2 (\TT)} \implies \|f\|_{ B^{\frac13}_{3,\infty} (\TT)} \le C \|f\|_{ B^{\frac56}_{2,\infty} (\TT)},
\]
constructing $U \in B^{\frac56}_{2,\infty}$ with a positive flux could be reached by replicating the interactions arising in $S$, while saturating the Bernstein inequality. 
This strategy is realised by a three-dimensional blurring of the skeleton frequencies, with the new amplitudes preserving solenoidality and being appropriately rescaled. To this end define the active regions (`blurs' of frequencies $\lambda_q F_0^{\Circled{i}}$ that are `active' in $S$) as follows:
	\begin{align*}
	A_{0,q}^{\Circled{1}} &\coloneqq \Zd \cap \lambda_{q} \left( F_0^{\Circled{1}} + [0,\epsilon]^3 \right), \\ 
	A_{0,q}^{\Circled{2}} &\coloneqq \Zd \cap \lambda_{q} \left( F_0^{\Circled{2}} + [0,\epsilon]^3 \right), \\ 
	A_{0,q}^{\Circled{3}} &\coloneqq \Zd \cap \lambda_{q} \left( F_0^{\Circled{3}} + [0,2\epsilon]^3 \right).
	\end{align*}
Define $\pi_{\xi}$ to be the orthogonal projection onto the plane $\xi^\bot$.
	Note that $ \pi_{\xi}(v) e^{i\xi\cdot x}$ is a divergence-free vector field and agrees with the classical Leray-Helmholtz projection of $v e^{i\xi\cdot x}$, since $\pi_\xi =  \id - \frac{\xi}{|\xi|}  \otimes \frac{\xi}{|\xi|} $.
	
	We can now define the components of $u_q=u_q^{\Circled{1}}+u_q^{\Circled{2}}+u_q^{\Circled{3}}$
	\begin{equation}\label{eq:u_q}
	\begin{aligned}
	u_q^{\Circled{1}} \coloneqq \epsilon^{-2} \lambda_q^{-\frac{7}{3}} \sum_{\xi\in A_{0,q}^{\Circled{1}}} \pi_{\xi} \left(  V_0^{\Circled{1}} \right) \sin \left( \xi\cdot x\right), \\
	u_q^{\Circled{2}} \coloneqq \epsilon^{-2} \lambda_q^{-\frac{7}{3}} \sum_{\xi\in A_{0,q}^{\Circled{2}}} \pi_{\xi} \left(  V_0^{\Circled{2}} \right) \cos \left( \xi\cdot x\right), \\
	u_q^{\Circled{3}} \coloneqq \epsilon^{-2} \lambda_q^{-\frac{7}{3}} \sum_{\xi\in A_{0,q}^{\Circled{3}}} \pi_{\xi} \left(  V_0^{\Circled{3}} \right) \cos \left( \xi\cdot x\right).
	\end{aligned}
	\end{equation}
	Finally, the (almost) planar vector field is $U_0\coloneqq \sum_{q\in3\NN} u_q$. It is real-valued and (weakly) divergence-free.
	\subsection{The rotations and the full field $U$}\label{ssec:full} 
The flux of $U_0$ is controlled only for $q\in3\NN$, thus merely a $\limsup$ result is possible. Straightforward 'condensation' $3\NN\to\NN$ results however in undesirable interactions. We overcome this problem by combining almost-planar fields at three distinct planes. Each of these three types of almost-planar fields is obtained by rotating the construction from Section~\ref{ssec:planar}.
	
	Fix the line $L\coloneqq \{x_1+2x_2=0=x_3\}$ within the plane $P_0$. Let $\rot$ denote the rotation around $L$ by $\frac{\pi}{3}$, which is given by the matrix
	\[ \rot = \frac{1}{10} \begin{bmatrix}
	9&-2&-\sqrt{15}\\-2&6&-2\sqrt{15}\\\sqrt{15}&2\sqrt{15}&5
	\end{bmatrix}\]
		The rotated components will be similar to $u_q$ from the previous section, but anchored instead at the plane $P_1=\rot P_0$ or $P_2=\rot^2 P_0$. More precisely: fix $q\in3\NN+1$, for $i=1,2,3$ define $V_1^{\Circled{i}} \coloneqq \rot V_0^{\Circled{i}}$, $F_1^{\Circled{i}} \coloneqq \rot F_0^{\Circled{i}}$ and the regions
	\begin{align*}
	A_{1,q}^{\Circled{1}} &\coloneqq \Zd \cap \lambda_{q} \left( F_1^{\Circled{1}} + [0,\epsilon]^3 \right), \\
	A_{1,q}^{\Circled{2}} &\coloneqq \Zd \cap \lambda_{q} \left( F_1^{\Circled{2}} + [0,\epsilon]^3 \right), \\
	A_{1,q}^{\Circled{3}} &\coloneqq \Zd \cap \lambda_{q} \left( F_1^{\Circled{3}} + [0,2\epsilon]^3 \right).
	\end{align*}
	(Note that we first rotate the skeleton frequencies and blur afterwards, since rotating the blurs $A_{0,q}^{\Circled{i}}$ results in non-integer values.) The components $u_q^{\Circled{i}}$ and $u_q$ are then defined as in Section~\ref{ssec:planar}, now with $V_1^{\Circled{i}}$ replacing $V_0^{\Circled{i}}$ and $A_{1,q}^{\Circled{3}}$ replacing $A_{0,q}^{\Circled{3}}$. We write $U_1 \coloneqq \sum_{q\in3\NN+1} u_q$. 
	
	The $\frac{2\pi}{3}$-rotated field $U_2\coloneqq \sum_{q\in3\NN+2} u_q$ is defined with $\rot $ in the definition of $V_1^{\Circled{i}}, F_1^{\Circled{i}}$ replaced by $\rot^2$ (yielding $V_2^{\Circled{i}}$, $F_2^{\Circled{i}}$, and $A_{2,q}^{\Circled{i}}$). Finally
	\[
	U=U_0+U_1+U_2.
	\]
For brevity, we will at times suppress the lower index $j$ in $A_{j,q}^{\Circled{i}}$ (denoting rotation), since it is unambiguously determined by $q\mod 3$. In other words
	\[
	U = \sum_{q} u_q = \sum_{q} u_q^{\Circled{1}}+u_q^{\Circled{2}}+u_q^{\Circled{3}}, \qquad \text{with} 
	\]
	\[
	\begin{aligned}
	u_q^{\Circled{1}} = \epsilon^{-2} \lambda_q^{-\frac{7}{3}} \sum_{\xi\in A_{q}^{\Circled{1}}} \pi_{\xi} \left(  V_\qm^{\Circled{1}} \right) \sin \left( \xi\cdot x\right) \\
	u_q^{\Circled{2}} = \epsilon^{-2} \lambda_q^{-\frac{7}{3}} \sum_{\xi\in A_{q}^{\Circled{2}}} \pi_{\xi} \left(  V_\qm^{\Circled{2}} \right) \cos \left( \xi\cdot x\right) \\
	u_q^{\Circled{3}} = \epsilon^{-2} \lambda_q^{-\frac{7}{3}} \sum_{\xi\in A_{q}^{\Circled{3}}} \pi_{\xi} \left(  V_\qm^{\Circled{3}} \right) \cos \left( \xi\cdot x\right).
	\end{aligned}
	\]
	\begin{figure}[h!]
		\caption{Active regions from three consecutive generations $q-1$, $q$, $q+1$ and support of $\nabla \varphi (\lambda_{q}^{-1} \cdot)$. Different colours of active regions indicate that in fact they are anchored at three different planes.}\label{fig:1}
		\centering
		\includegraphics[ width=10cm]{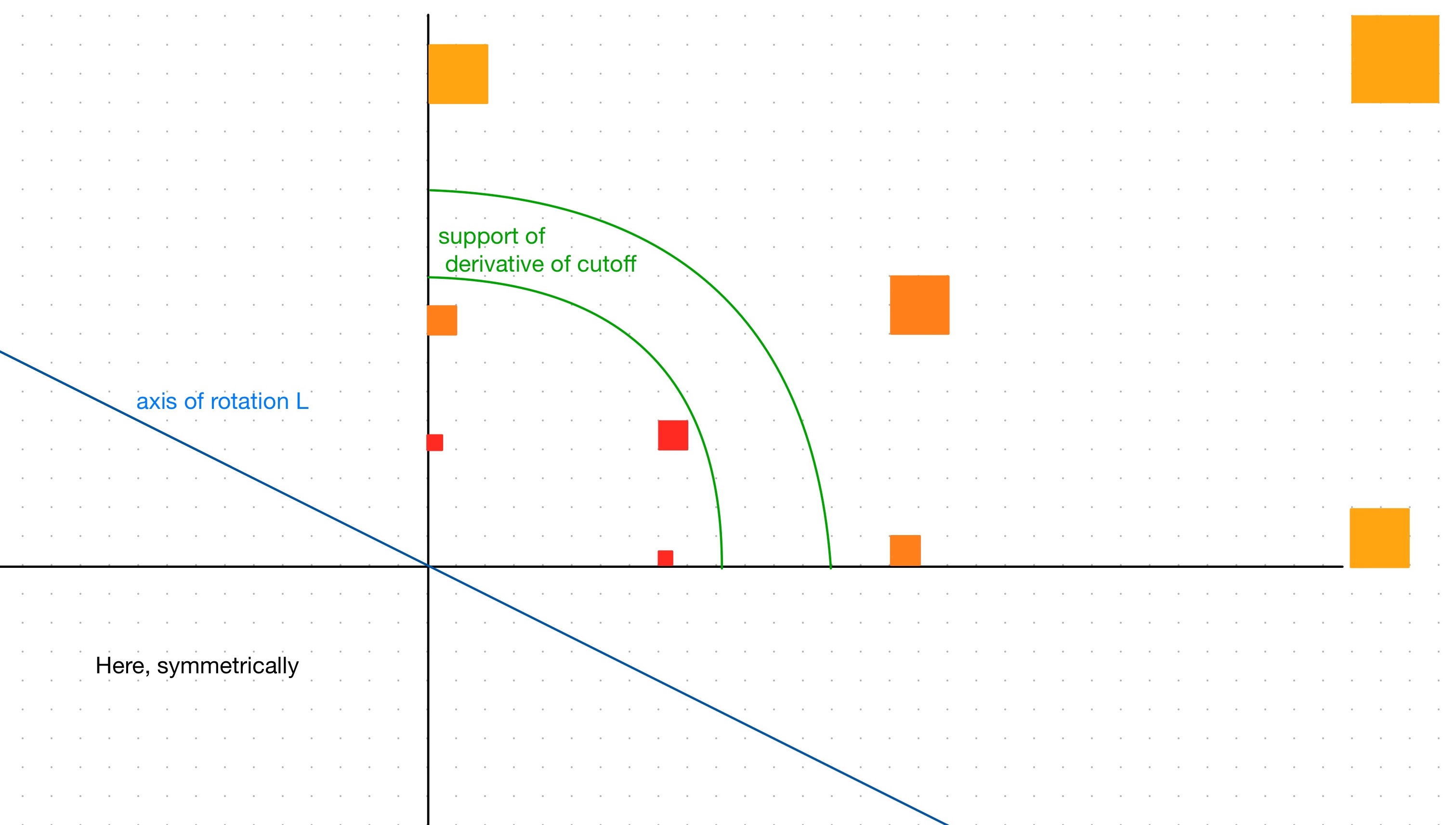}
	\end{figure}
	
	\subsection{Regularity of $U$}
	By $O(\epsilon)$ we will denote a {real} scalar, vector or tensor of magnitude $\lesssim \epsilon$, i.e.\ such that $|O(\epsilon)| \le C \epsilon$, with a uniform constant $C$ (i.e.\ independent from any parameters).
	
	For objects defined in Sections~\ref{ssec:planar},~\ref{ssec:full} we have
	\begin{prop}\label{prop:prop}
		For $j=0,1,2$ (rotations) and $i=1,2,3$ it holds
	\begin{equation}
		\pi_\xi \left( V_j^{\Circled{i}} \right) - V_j^{\Circled{i}} = O(\epsilon) \qquad  \forall \, \xi\in A_{j,q}^{\Circled{i}}, \label{projection}
	\end{equation}
	\begin{subequations}\label{blur-size}
	\begin{equation}
	|A_{j,q}^{\Circled{i}}| \in [(\epsilon\lambda_{q}-1)^3, (\epsilon\lambda_{q}+1)^3] \quad \text{ for } \quad i=1,2,
	\end{equation}	\begin{equation}
	|A_{j,q}^{\Circled{3}}| \in [(2\epsilon\lambda_{q}-1)^3, (2\epsilon\lambda_{q}+1)^3]. 
		\end{equation}
		\end{subequations}
	Moreover, for any $p \in (1, \infty]$
		\begin{align}\label{regularity_comp}
		\|u_q\|_{L^p} &\le C_p \epsilon^{1-\frac{3}{p}} \lambda_{q}^{\frac23- \frac{3}{p}}, \\
		\label{regularity_full}
		\|U\|_{\dot{B}^{\frac{3}{p}-\frac{2}{3}}_{p,\infty}} &\le C_p \epsilon^{1-\frac{3}{p}}.
		\end{align}
	\end{prop}
	\begin{proof}
		From the construction in Sections~\ref{ssec:planar},~\ref{ssec:full} we have
		\begin{equation}
		\label{blur-is-small}
		\xi \in A_{j,q}^{\Circled{i}} \implies \frac{\xi}{|\xi|} = \frac{F_j^{\Circled{i}}}{| F_j^{\Circled{i}}|}+ O(\epsilon)
		\end{equation}
		and $F_j^{\Circled{i}} \perp V_j^{\Circled{i}}$. The projection $\pi_\xi =  \id - \frac{\xi}{|\xi|}  \otimes \frac{\xi}{|\xi|} $, therefore
		\[ \pi_\xi \left( V_j^{\Circled{i}} \right)
		= V_j^{\Circled{i}} - \frac{\xi}{|\xi|} \left(\frac{F_j^{\Circled{i}}}{| F_j^{\Circled{i}}|}+O(\epsilon)\right)\cdot V_j^{\Circled{i}} = V_j^{\Circled{i}} - \frac{\xi}{|\xi|} O(\epsilon) \]
		from which \eqref{projection} follows immediately. 
		
In any cube $x_0+[0,\epsilon\lambda_{q}]^3$, there are at most $\left( \lfloor{\epsilon\lambda_{q}} \rfloor  +1\right)^3\le (\epsilon\lambda_{q}+1)^3$ and at least $\left( \lfloor{\epsilon\lambda_{q}} \rfloor \right)^3 \ge (\epsilon\lambda_{q}-1)^3$ lattice points. This proves \eqref{blur-size} in the case $i=1,2$. For $i=3$ the only difference is that the cube has twice longer edges.
		
		Since $\pi_\xi$ agrees with the Leray-Helmholtz projector, we can use the $L^p$ theory for the latter, obtaining
		\[ \begin{aligned}
		\|u_q^{\Circled{1}} \|_p &\le   C_p \epsilon^{-2} \lambda_q^{-\frac{7}{3}} \| \sum_{\xi\in A_{q}^{\Circled{1}}} V_\qm^{\Circled{1}} \sin \left( \xi\cdot  \right) \|_p \\
		&\le  C_p \epsilon^{-2} \lambda_q^{-\frac{7}{3}} \| \sum_{\xi\in A_{q}^{\Circled{1}}} \sin \left( \xi\cdot \right) \|_p \\
		&\le C_p \epsilon^{-2} \lambda_q^{-\frac{7}{3}} (\epsilon \lambda_q)^{3(1-\frac{1}{p})},
		\end{aligned} \]
		where the last but one inequality holds, because $V_\qm^{\Circled{1}}$ is constant. The last inequality, since $A_{q}$ is a cube, follows along the lines of the Dirichlet kernel estimate. An analogous computation for $u_q^{\Circled{2}}$, $u_q^{\Circled{3}}$ yields  \eqref{regularity_comp} for $p \in (1, \infty)$.
		The case $p = \infty$ follows from the (non-optimal) straightforward bound $\|u_q\|_{L^\infty} \le \|\hat{u}_q\|_{L^\infty} \spt (\hat{u}_q)$ and application of \eqref{projection} and \eqref{blur-size}.
		
		The final regularity statement \eqref{regularity_full} follows from \eqref{regularity_comp} and the fact that by \eqref{spt-u_qj} below at most 3 components $u_q$ interfere with a single  Littlewood-Paley shell.
	\end{proof}
	
	\subsection{Fourier side of $U$} 
	For the sake of the next section, let us write $u_q$ in terms of Fourier modes, restricting our attention to the leading order amplitudes, provided by \eqref{projection}:
	\begin{equation}
	\label{uq-fourier}
	\begin{aligned}
	\widehat{u_q^{\Circled{1}}} &= -i\lambda_{q}^{-\frac{7}{3}} \epsilon^{-2} \left( V_{\qm}^{\Circled{1}}+O(\epsilon) \right) \left( \chi_{A_{q}^{\Circled{1}}} - \chi_{-A_{q}^{\Circled{1}}} \right), \\
	\widehat{u_q^{\Circled{2}}} &= \lambda_{q}^{-\frac{7}{3}} \epsilon^{-2} \left( V_{\qm}^{\Circled{2}}+O(\epsilon) \right) \left( \chi_{A_{q}^{\Circled{2}}} + \chi_{-A_{q}^{\Circled{2}}} \right), \\
	\widehat{u_q^{\Circled{3}}} &= \lambda_{q}^{-\frac{7}{3}} \epsilon^{-2} \left( V_{\qm}^{\Circled{3}}+O(\epsilon) \right) \left( \chi_{A_{q}^{\Circled{3}}} + \chi_{-A_{q}^{\Circled{3}}} \right),
	\end{aligned}
	\end{equation}
	where $\chi_A$ denotes the characteristic function of $A$. Differentiation yields
	\begin{equation}
	\begin{aligned}
	\widehat{\nabla u_q^{\Circled{1}}} &= -\lambda_q^{-\frac{4}{3}} \epsilon^{-2} \left( V_\qm^{\Circled{1}} \otimes F_\qm^{\Circled{1}} +O(\epsilon) \right) \left( \chi_{A_{q}^{\Circled{1}}} + \chi_{-A_{q}^{\Circled{1}}} \right) ,\\
	\widehat{\nabla u_q^{\Circled{2}}} &= i\lambda_q^{-\frac{4}{3}} \epsilon^{-2} \left( V_\qm^{\Circled{2}} \otimes F_\qm^{\Circled{2}} +O(\epsilon) \right) \left( \chi_{A_{q}^{\Circled{2}}} - \chi_{-A_{q}^{\Circled{2}}} \right) ,\\
	\widehat{\nabla u_q^{\Circled{3}}} &= i\lambda_q^{-\frac{4}{3}} \epsilon^{-2} \left( V_\qm^{\Circled{3}} \otimes F_\qm^{\Circled{3}} +O(\epsilon) \right) \left( \chi_{A_{q}^{\Circled{3}}} - \chi_{-A_{q}^{\Circled{3}}} \right) .
	\end{aligned}
	\label{uq-der-1}
	\end{equation}
	Observe that the amplitudes of $\nabla u_q^{\Circled{1}}$ are real while the other ones are purely imaginary.
	
	We will also need the following precise estimates about the shells in which the Fourier support of $U$ is contained:
		\begin{equation}
		\begin{aligned}
		\spt (\widehat{u_q^{\Circled{1}}}) &\subset \{ \lambda_{q} (1-\sqrt{3}\epsilon) \le |\xi| \le \lambda_{q} (1+\sqrt{3}\epsilon)\},\\
		\spt (\widehat{u_q^{\Circled{2}}}) &\subset \{ \lambda_{q} (2-\sqrt{3}\epsilon) \le |\xi| \le \lambda_{q} (2+\sqrt{3}\epsilon)\}, \\
		\spt (\widehat{u_q^{\Circled{3}}}) &\subset \{ \lambda_{q} (\sqrt{5}-2\sqrt{3}\epsilon) \le |\xi| \le \lambda_{q} (\sqrt{5}+2\sqrt{3}\epsilon)\}.
		\end{aligned}
		\label{spt-u_qj-region}    
		\end{equation}
		Hence
		\begin{equation}
		\spt (\widehat{u_q}) \subset \{ \lambda_{q} (1-2\epsilon) < |\xi| < \lambda_{q} (\sqrt{5}+4\epsilon)\}.
		\label{spt-u_qj}    
		\end{equation}
Bounds \eqref{spt-u_qj-region} can be seen using Figure~\ref{fig:1}. Let us justify them more extensively. It holds
		$\spt (\widehat{u_q^{\Circled{i}}}) = \pm A_{q}^{\Circled{i}}$; in particular the supports are symmetric, so we can ignore the `negative' part in our computation. For  $i=1,2$, $A_{q}^{\Circled{i}} \subset Q^i:=\lambda_{q} ( F_\qm^{\Circled{i}} + [0,\epsilon]^3)$, while for $Q^3$ we replace $\epsilon$ with $2\epsilon$. The diameter of $Q^i$ equals $\sqrt{3} \lambda_q \epsilon$ for $i=1,2$ and $2\sqrt{3}\lambda_q \epsilon$ for $i=3$.
		Observe that 
		\[ |F_\qm^{\Circled{1}}|=1, \ |F_\qm^{\Circled{2}}|=2, \ |F_\qm^{\Circled{3}}|=\sqrt{5} \]
		irrespective of rotation. This, the fact that $\lambda_q F_\qm^{\Circled{i}} \in Q^i$ and the values of the diameters yield the generous estimate \eqref{spt-u_qj-region}.
		
			\section{Flux decomposition}\label{sec:decomposition}
	Let $v_1, v_2, v_3$ be any vector fields from $\Td$ to $\Cd$. Then, for $\xi_j \in \Zd$, $j=1,2,3$, via Parseval
	\begin{equation}\label{eq:fourier-flux}
		\begin{gathered}
	\int (v_1\otimes v_2) :\nabla v_3 
	=\sum_{\xi_1 + \xi_2 + \xi_3 = 0} \widehat{v_1}(\xi_1) \otimes \widehat{v_2}(\xi_2) : \widehat{\nabla v_3}(\xi_3).
	\end{gathered}
	\end{equation}
	Thus, any contribution to the integral necessitates $\xi_1+\xi_2+\xi_3=0$. This observation is an underlying principle for the following analysis of the flux. In particular, modes satisfying $\xi_1+\xi_2+\xi_3=0$ will be referred to as 'interacting' (i.e.\ contributing to the flux), exhibiting 'non-zero interactions' etc.\

	The aim of this section is to split the flux $\Pi_q (U)$ into a local part $\piloc$, which is restricted to interactions within a single component $u_q$, and a non-local part $\pinonloc$, which involves modes further apart. 
	We will show that $\piloc$ behaves precisely like the flux of the skeleton field $S$. However, unlike in the $\Pi_q (S)$ case, $\pinonloc$ will contribute to the flux. Indeed, a blur of a very high frequency interacts with its almost-symmetric counterpart, producing a much lower non-zero mode, which in turn may interact with a low-frequency mode. 
	
	Recall that $|O(\epsilon)| \le C \epsilon$ with a uniform $C$. We will from now onwards tacitly assume that $\epsilon \le \epsilon_0$, where $C \epsilon_0 \le \frac{1}{C}$. Consequently, $1 - |O(\epsilon)| \ge 1- \frac{1}{C}$. Even though $O(\epsilon)$ is unsigned and even not necessarily a real number, we will write at times $1-O(\epsilon)$, meaning $1-|O(\epsilon)|$. Denoting $a \odot b:= a \otimes b + b \otimes a$, we state
	\begin{lem}[Flux decomposition]
		\label{prop:decomp}
		The energy flux $\Pi_{q} (U)= \piloc + \pinonloc$, where
		\begin{equation}
		\label{Pi^local2}
		\piloc =  \int u_{q}^{\Circled{2}} \odot u_{q}^{\Circled{3}} : \nabla u_{q}^{\Circled{1}}
		\end{equation}
and 
		\begin{equation}\label{Pi^nonloc2}
		\pinonloc = \sum_{i=1,2,3} \sum_{l\le q}\sum_{k>max(q, l+ N_\epsilon)} \int u_k^{\Circled{i}} \otimes u_k^{\Circled{i}} : \nabla u_l^{\Circled{1}}
		\end{equation}
with $N_\epsilon = \lfloor 3- \log\epsilon \rfloor$. 
	\end{lem}
	The remainder of this section contains proof of Lemma~\ref{prop:decomp}.
	\subsection{Initial splitting of $\Pi_q (U)$}
	Let us define for $U = \sum_{q} u_q$ of Section~\ref{ssec:full}
	\[ U_{<q} \coloneqq \sum_{l<q} u_{l} \qquad \text{ and analogously: } \quad U_{>q},\ U_{\le q},\ U_{\ge q}. \]
	The splitting $ U = U_{<q} + u_{q} + U_{>q}$ is usually called Bony decomposition.
	By \eqref{spt-u_qj-region} and the choice of the cutoff function $\varphi$, implying there are no active modes of $U$ in the support of  derivative of $\varphi$, we have (cf.\ Figure \ref{fig:1})
	\begin{equation}\label{eq:io} S_{q}(U) =   U_{<q}+ u_q^{\Circled{1}} \quad \text{ and } \quad U- S_{q}(U) =   u_q^{\Circled{2}} + u_q^{\Circled{3}} +  U_{>q}.  \end{equation}
Therefore
	\begin{equation}\label{no-cutoff}
	\begin{aligned}
	\Pi_{q}(U) &= \int S_{q}(U\otimes U) : \nabla S_{q}(U) =  \int  S_{q}(u \otimes U) : \nabla (U_{<q}+ u_q^{\Circled{1}} ) \\
	&=  \int (U\otimes U) : \nabla (U_{<q}+ u_q^{\Circled{1}} ).
	\end{aligned}
	\end{equation}
	$U$ is divergence-free and integrable, while $U_{<q}+ u_q^{\Circled{1}}$ is smooth, thus integrating by parts
	\begin{equation}\label{eq:van}
	\int (U_{<q} + u_q^{\Circled{1}}) \otimes U : \nabla (U_{<q} + u_q^{\Circled{1}}) =0.
	\end{equation}
	Subtracting  \eqref{eq:van} from \eqref{no-cutoff} we obtain
	\begin{align}
	\Pi_{q}(U) =\int  \left(u_q^{\Circled{2}} + u_q^{\Circled{3}} + U_{>q} \right) \otimes U : \nabla \left(U_{<q}+ u_q^{\Circled{1}}  \right).
	\label{symmetry-Piqj}
	\end{align}
	We split now $\Pi_{q}(U)$ of \eqref{symmetry-Piqj} into the part $\pinonloc$ where modes are separated and the remainder $\piloc$ as follows:
	
	\begin{equation}
	\label{Pi^local}
	\piloc = \int  (u_q^{\Circled{2}} + u_q^{\Circled{3}} + u_{q+1}) \otimes (u_{q-1}+ u_{q} + u_{q+1}) : \nabla (u_{q-1}+ u_q^{\Circled{1}}),
	\end{equation}
	\begin{equation} \label{Pi^nonloc}
	\begin{aligned}
	\pinonloc &= \int U_{> q+1}  \otimes U : \nabla \left(U_{<q}+ u_q^{\Circled{1}} \right) \\
	&+ \int  (u_q^{\Circled{2}} + u_q^{\Circled{3}} + u_{q+1} ) \otimes U : \nabla U_{< q-1} \\
	&+ \int  (u_q^{\Circled{2}} + u_q^{\Circled{3}} + u_{q+1} ) \otimes (U_{< q-1}
	+U_{> q+1}) : \nabla (u_{q-1}+ u_q^{\Circled{1}}).
	\end{aligned}
	\end{equation}
	We need to show that \eqref{Pi^local} is in fact \eqref{Pi^local2}, and that \eqref{Pi^nonloc} is in fact \eqref{Pi^nonloc2}.
	\subsection{Exclusion of three different planes interacting}		
	We will say that a point $\xi$ is anchored to a set $\Omega$ if $dist(\xi, \Omega) \le 2\epsilon|\xi|$.  We claim that each frequency $\xi$ of $\widehat U$ is unambiguously anchored at one of three different planes $P_0$, $P_1$ or $P_2$, thanks to $\epsilon \le \epsilon_0$ small in our definition of active frequency regions $A_{q}^{\Circled{i}}$. 
	
	Indeed, $\lambda_q F_\qm^{\Circled{i}} \in P_\qm$, so $dist(\xi,  P_\qm) \le |\xi-\lambda_q F_\qm^{\Circled{i}}|$. At the same time, if $\xi\in A_q^{\Circled{i}}$ then by construction $\xi$ belongs to a cube with a vertex $\lambda_q F_\qm^{\Circled{i}}$ and with its side length $\epsilon \lambda_q $ or $2 \epsilon \lambda_q $. This reads for $\xi\in A_q^{\Circled{i}}$ with $i=1,2$
	\[
dist(\xi,  P_\qm) \le	|\xi-\lambda_q F_\qm^{\Circled{i}}| \le \sqrt{3}\epsilon\lambda_q \le  \frac{\sqrt{3}\epsilon}{1-2\epsilon} |\xi|,
		\]
	where the last inequality stems from  \eqref{spt-u_qj}. Hence for $i=1,2$, $\xi\in A_q^{\Circled{i}}$ are anchored to $P_\qm$, provided $\frac{\sqrt{3}}{1-2\epsilon_0} \le 2$. Similarly for $\xi\in A_q^{\Circled{3}}$, because
    \[ dist(\xi,P_\qm) \le |\xi-\lambda_q F_\qm^{\Circled{3}}| \le 2\sqrt{3}\epsilon \lambda_q \le \frac{2\sqrt{3}\epsilon}{\sqrt{5}-2\sqrt{3}\epsilon} |\xi|, \]
 applying now \eqref{spt-u_qj-region} and again an uniform upper bound on $\epsilon_0$.
 
 Since $\epsilon$ is small, we see that the anchoring is unambiguous, as claimed.

	Our next goal is to exclude any interactions between modes related to three different planes $P$.
	\begin{prop}\label{prop:windmill}
		If each $\spt (\widehat{{u}_{q'}})$, $\spt (\widehat{{u}_{q''}})$, $\spt (\widehat{{u}_{q'''}})$ is anchored at a different plane $P$, then
		\begin{equation}\label{eq:no3}
		\int  u_{q'} \otimes u_{q''} : \nabla u_{q'''} =0.
		\end{equation}
	\end{prop}
	\begin{proof}
		Every $P$ stems from a rotation around axis $L$, whose normal plane we call $L^\perp$. Let us denote the orthogonal projection onto $L^\perp$ by $\pi_{L^\perp}$. For modes $\xi_j \in \spt (\widehat{{u}_{q_j}})$ with $j=1,2,3$ to interact, it is necessary that $\pi_{L^\perp} (\xi_i)$ do interact, i.e.\ that 
		\begin{equation}\label{eq:perpint}
		\pi_{L^\perp} (\xi_1) + \pi_{L^\perp} (\xi_2) + \pi_{L^\perp} (\xi_3) = 0.
		\end{equation}
		For the following considerations, it may be beneficial to visualise $U$ projected onto $L^\perp$, i.e. Figure \ref{fig:1} seen along the axis $L$, see Figure \ref{fig:2}. 
		\begin{figure}[h]
			\caption{Windmill: dots indicate (approximately) active regions projected on the plane of rotation $L^\perp$.}\label{fig:2}
			\centering
			\includegraphics[ width=8cm]{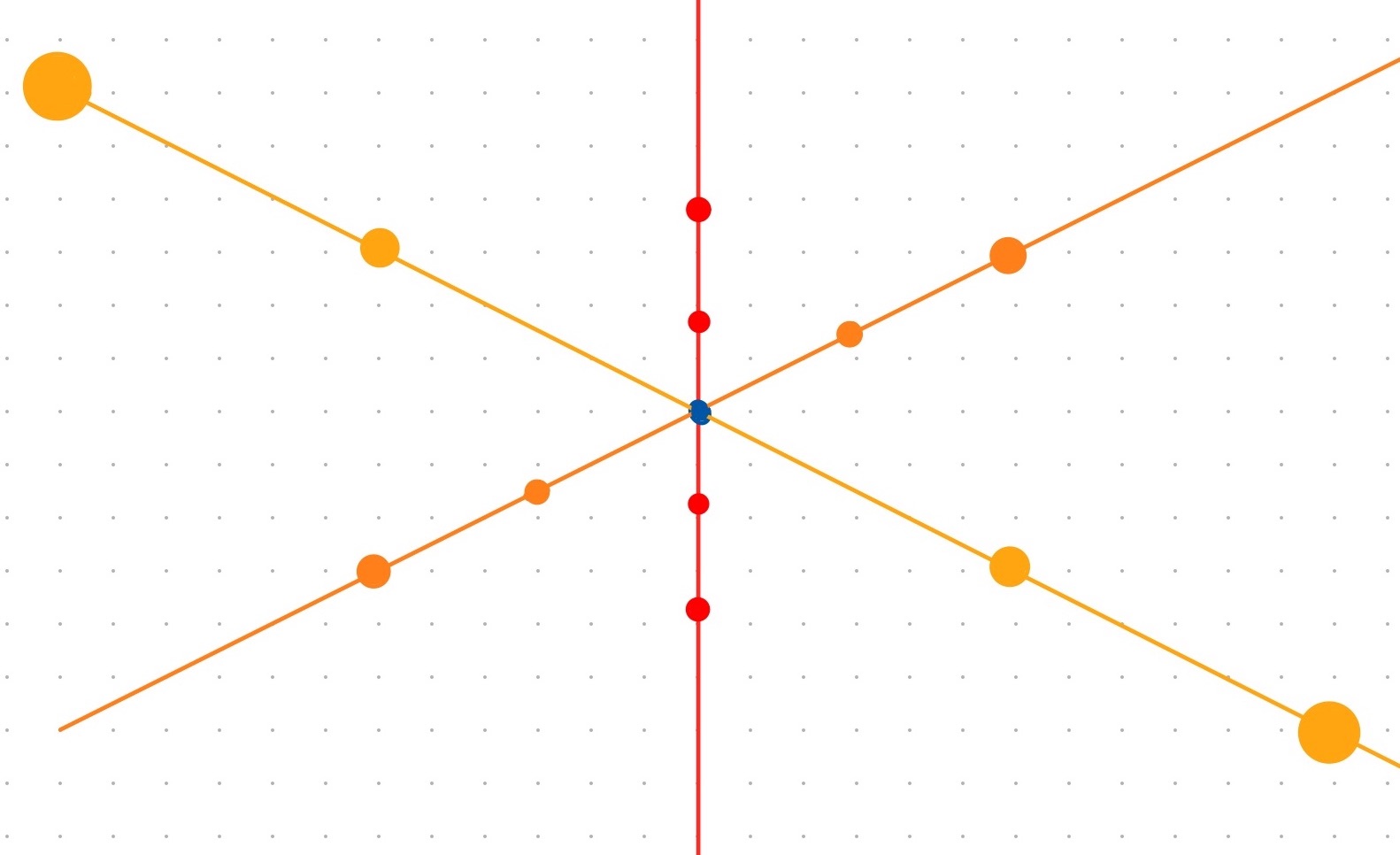}
		\end{figure}

Observe that $\pi_{L^\perp}$ projections of modes of $u_q$ have norms equal either $\frac{2\sqrt5}{5} \lambda_q (1+ O(\epsilon)) =:\tilde \lambda_q (1+O(\epsilon))$ or $2 \tilde \lambda_q (1+O(\epsilon))$.

We assumed that each mode $\xi_i$ is anchored at a different $P$ plane. Consequently, each $\pi_{L^\perp} (\xi_i)$ is anchored to a different line $l_i$ which is $\pi_{L^\perp}$ projection of a plane $P$.  Take the mode with the largest norm, say $\xi_1$ of $\widehat{u_{q_0}}$, anchored at $l_1$. We have 
		\begin{subequations}
			\begin{equation}\label{eq:x1}
			\text{either}\quad |\pi_{L^\perp} (\xi_1)| =\tilde \lambda_{q_0}(1 + O(\epsilon)),
			\end{equation}    
			\begin{equation}\label{eq:x1l}
			\text{or}\qquad |\pi_{L^\perp} (\xi_1)| =2\tilde \lambda_{q_0}(1 + O(\epsilon)).
			\end{equation}
		\end{subequations}
		The remaining two active $\pi_{L^\perp} (\xi_i)$'s, $i=2,3$  are anchored, respectively, at different lines $l_2$, $l_3$.
		Therefore we have
		\begin{subequations}\label{eq:windmill}
			\begin{equation}\label{eq:windmill2}
			|\pi_{L^\perp} (\xi_2)| =  2^{-j}  \tilde\lambda_{q_0}(1+O(\epsilon)) \quad \text{ for a } j \ge 0 
			\end{equation}    
			\begin{equation}\label{eq:windmill3}
			|\pi_{L^\perp} (\xi_3)| =  2^{-k-1}\tilde \lambda_{q_0}(1+O(\epsilon)) \quad \text{ for a } k \ge 0.
			\end{equation}
		\end{subequations}
		The first consequence of \eqref{eq:windmill} is that \eqref{eq:x1l} cannot happen. Indeed, using \eqref{eq:perpint} we have 
		\begin{equation}\label{eq:windmillex1}
		|\pi_{L^\perp} (\xi_1)| = |\pi_{L^\perp} (\xi_2) + \pi_{L^\perp} (\xi_3)| 
		\le \tilde\lambda_{q_0}(1+O(\epsilon)) + 2^{-1} \tilde\lambda_{q_0}(1+O(\epsilon)) <  \frac32 \tilde\lambda_{q_0}(1+O(\epsilon)).
		\end{equation}
		Next, knowing that  \eqref{eq:x1} is the only possibility, we argue that  \eqref{eq:x1}  requires $j=1$ in \eqref{eq:windmill2}. Indeed, $j<1$ is immediately excluded analogously to \eqref{eq:windmillex1}. Take case  $j=0$ in \eqref{eq:windmill2}. The involved projections of modes $\pi_{L^\perp} (\xi_1)$, $\pi_{L^\perp} (\xi_2)$ are anchored at different lines $l_1$, $l_2$, with angle  $\pi/3$ between them. This spreading means that, by adding vectors we have \[|\pi_{L^\perp} (\xi_1) +\pi_{L^\perp} (\xi_2)| \ge \lambda_{q_0}(1-O(\epsilon)),\]
		but then \eqref{eq:windmill3} makes \eqref{eq:perpint} impossible. (Recall we write $1-O(\epsilon)$ for $1-|O(\epsilon)|$.)

		Finally, having both \eqref{eq:x1} and  \eqref{eq:windmill2} with $j=1$, we see that $k>0$ in \eqref{eq:windmill3} is now excluded analogously to \eqref{eq:windmillex1}.
		
		Altogether, \eqref{eq:perpint}
	may occur only if
		\[
		|\pi_{L^\perp} (\xi_1)| =\tilde \lambda_{q_0}(1+O(\epsilon)), \  |\pi_{L^\perp} (\xi_2)|= \frac{\tilde \lambda_{q_0}}{2} (1+O(\epsilon)), \  |\pi_{L^\perp} (\xi_3)| = \frac{\tilde \lambda_{q_0}}{2} (1+O(\epsilon)),
		\]
		but due to $\pi/3$ angle between $l_2$, $l_3$,  $|\pi_{L^\perp} (\xi_2) + \pi_{L^\perp} (\xi_3)| \le \frac{\sqrt{3}}{2} \tilde \lambda_{q_0} (1+ O(\epsilon))<\tilde \lambda_{q_0}(1 -O(\epsilon))$, so \eqref{eq:perpint} is impossible.
	\end{proof}
	
	\subsection{Exclusion of near-field interactions}
	It turns out that active modes which are not all from the same $\widehat{u_q}$ may interact only in the following scenario: the two higher-frequency modes are from the same $\widehat{u_q}$, while the third mode is of  considerably lower frequency.
	More precisely, we have
	\begin{prop}\label{prop:nowhere-near}
	Let $q_1\le q_2\le q_3$. Assume that $u_{q_j}$'s interact, i.e.~there are $\xi_1, \xi_2, \xi_3$ where $\xi_j\in \spt \left( \widehat{u_{q_j}} \right)$ for $j=1,2,3$ such that
	\[ \xi_1+\xi_2+\xi_3 = 0. \]
	Suppose $q_1<q_3$, then $q_1+3\le q_2=q_3$.
	\end{prop}
	
	\begin{proof}
	We distinguish three elementary cases, depending on the separation between the largest $q_3$ and the smallest $q_1$.
	
	\emph{Case 1.} 
	$q_1=q_3-1$. Since $q_2\in \{q_1,q_3\}$ there are exactly two modes anchored at the same plane and the third mode at a different one, say $\xi,\eta$ are anchored at $P_0$ while $\zeta$ is not. Then, on the one hand
	\begin{equation}\label{eq:p3oh} dist(\xi+\eta,P_0) \le dist(\xi,P_0) + dist(\eta,P_0) \le 2\epsilon (|\xi|+|\eta|) \le 2 \epsilon (2\sqrt{5}+O(\epsilon)) \lambda_{q_3}, \end{equation}
	with the last inequality given by \eqref{spt-u_qj} and the fact that $q_3$ is the largest of $q_j$'s.
	On the other hand, projections do not increase distances, so $ dist(\zeta,P_0) \ge  dist(\pi_{L^\perp}(\zeta),l_0)$, with $l_0=\pi_{L^\perp} (P_0)$ (recall setting of proof of Proposition~\ref{prop:windmill}). 
	Since the angle between the planes $P_\qm$ is always $\frac{\pi}{3}$ (thus between lines $l_\qm$ being their projections) and $\zeta$ is anchored to  $l_\qm \neq l_0$, we have, cf Figure \ref{fig:2}
		\begin{equation}\label{eq:p3ah}  
		\begin{aligned}
		dist(\zeta,P_0) &\ge  dist(\pi_{L^\perp}(\zeta),l_0) \ge \tilde \lambda_{q_1}(1-O(\epsilon)) ( \sin(\frac{\pi}{3})- O(\epsilon) ) \\
		&\ge ( \frac{\sqrt{15}}{5} - O(\epsilon) )  \lambda_{q_1} = ( \frac{\sqrt{15}}{10} - O(\epsilon) )  \lambda_{q_3}
		\end{aligned}
		\end{equation}
Comparing \eqref{eq:p3oh} and \eqref{eq:p3ah}, we see that $dist(\xi+\eta,P_0)< dist(\zeta,P_0)$ for $\epsilon\le\epsilon_0$, so there is no interaction.
	
	\emph{Case 2.}	$q_1=q_3-2$. There are two possibilities: 
	
	a) $q_2=q_1+1$, i.e.\ $q_2$ is strictly inbetween $q_1$ and $q_3$. Then all $\xi_j$ are anchored at a different plane, so by Proposition~\ref{prop:windmill} there is no interaction.
	
b) $q_2\in \{q_1,q_3\}$. Then the very same argument as in Case~1 applies, with a smaller but fixed parameter $\epsilon_0$.
	
	\emph{Case 3.}
	$q_1\le q_3-3$. We can estimate using \eqref{spt-u_qj} twice
	\[ (\sqrt{5} + 4\epsilon) \lambda_{q_2} > |\xi_2| = |\xi_1+\xi_3| \ge |\xi_3|-|\xi_1| \ge (1-\frac{\sqrt{5}}{8} -\frac{5\epsilon}{2}) \lambda_{q_3} > \frac{\sqrt{5}+4\epsilon}{4} \lambda_{q_3}, \]
	where the last inequality uses $\epsilon\le\epsilon_0$. Comparing the leftmost and the rightmost quantity above, we see that necessarily $q_2\ge q_3-1$. We have thus two options: either $q_2+1=q_3$ or $q_2=q_3$; the latter being the final statement. Therefore it remains to exclude $q_2+1=q_3$. In such case $\xi_2$ and $\xi_3$ are anchored at different planes, say $P_0$ and $P_1$, respectively. Analogously to \eqref{eq:p3ah} we thus have
	\[ dist(\xi_3,P_0) \ge ( \frac{\sqrt{15}}{5} - O(\epsilon) )  \lambda_{q_3}.  \]
The mode $\xi_2$ is anchored at the plane $P_0$, so $-\xi_2$ is also anchored at the plane $P_0$. Consequently, via triangle inequality
\[
 |\xi_2+\xi_3| + 2 \epsilon |\xi_2| \ge dist(\xi_3,P_0).
\]
Together, the above two inequalities yield
\[ |\xi_2+\xi_3|\ge dist(\xi_3,P_0) - 2 \epsilon |\xi_2| \ge (\frac{\sqrt{15}}{5} - O(\epsilon)) \lambda_{q_3} > (\frac{\sqrt{5}}{8} + \frac{\epsilon}{2}) \lambda_{q_3} \ge |\xi_1|, \]
with the last inequality due to  \eqref{spt-u_qj} and $q_1\le q_3-3$. Hence we excluded  $q_2+1=q_3$.
	\end{proof}
		\subsection{Proof of \eqref{Pi^local2}}
	Observe that
	\begin{equation}\label{eq:locA}
	\begin{aligned}
	\piloc &= \int  (u_q^{\Circled{2}} + u_q^{\Circled{3}} + u_{q+1}) \otimes (u_{q-1}+ u_{q} + u_{q+1}) : \nabla (u_{q-1}+ u_q^{\Circled{1}}) \\
	&= \int\left( u_{q}^{\Circled{2}} + u_{q}^{\Circled{3}} \right) \otimes u_q : \nabla u_{q}^{\Circled{1}}.
	\end{aligned}
	\end{equation}
	The former identity of \eqref{eq:locA} is  \eqref{Pi^local}; the latter holds thanks to  Proposition~\ref{prop:nowhere-near}. Indeed, the remaining terms involve at least two different $q_j$'s,  but they do not satisfy the separation condition $q_1 + 3 \le q_3$ (here $q_1=q-1$, $q_3=q+1$) demanded in Proposition~\ref{prop:nowhere-near} for a non-zero interaction.
	
	Let us now take $\xi^1, \xi^2$ in the Fourier support of the tensor product terms in \eqref{eq:locA}. For a nonzero interaction necessarily $\xi^1+\xi^2\in \spt (\widehat{{u}_{q}^{\Circled{1}}})$, which requires by construction (consult Figure \ref{fig:1}) one of  $\xi^1, \xi^2$ in $\spt (\widehat{{u}_{q}^{\Circled{2}}})$ and the other in $\spt (\widehat{{u}_{q}^{\Circled{3}}})$, implying	\eqref{Pi^local2}.
	
	\subsection{Proof of  \eqref{Pi^nonloc2}}
First observe that none of the integrals in \eqref{Pi^nonloc} involves only one $u_{q}$, consequently Proposition~\ref{prop:nowhere-near} applies ($q_1<q_3$). Hence, any non-zero interaction within $\pinonloc$ has the form
\[ \Pi_{l,k} \coloneqq \int u_k\otimes u_k : \nabla u_l, \quad \text{ where } \; l\le q\le k \;\text{ and }\; l+3\le k. \]
Let us analyse $\Pi_{l,k}$. First note that $\xi_1+\xi_2+\xi_3=0$ together with $\xi_1,\xi_2 \in \spt (\widehat{u_k})$ and $\xi_3 \in \spt (\widehat{u_l})$ imply that $\xi_1\in A_k^{\Circled{i}}$ and $\xi_2\in -A_k^{\Circled{i}}$ with the same $i$, or $\xi_1\in- A_k^{\Circled{i}}$ and $\xi_2\in +A_k^{\Circled{i}}$ with the same $i$. Indeed, in any other case, by the separation of active regions $\pm A_k^{\Circled{i}}, \pm A_k^{\Circled{j}}$ with $i\ne j$ (cf Figure \ref{fig:1}) one has
\[|\xi_1+\xi_2| \ge (1-O(\epsilon))\lambda_k =  8(1-O(\epsilon))\lambda_l > |\xi_3|,\] 
with the last inequality via \eqref{spt-u_qj}, so there cannot be any interaction. Therefore
\[\Pi_{l,k}= \sum_{i=1,2,3} \int u_k^{\Circled{i}} \otimes u_k^{\Circled{i}} : \nabla u_l = \sum_{i=1,2,3} \int u_k^{\Circled{i}} \otimes u_k^{\Circled{i}} : \nabla u_l^{\Circled{1}},  \]
where the last identity is due to \eqref{eq:fourier-flux} and the following observation: the product of Fourier modes from opposite regions is real by \eqref{uq-fourier}, but by \eqref{uq-der-1} modes in $\nabla u_q^{\Circled{i}}$ are purely imaginary unless $i=1$. The total flux is real, so any imaginary components are irrelevant.

Summarising, the non-local flux reduces to
\[\pinonloc = \sum_{l\le q} \sum_{k\ge \max \{q,l+3\}} \Pi_{l,k}. \]
Using the expression for $\Pi_{l,k}$ we can strip it down even more.
Since $\xi_1,\xi_2$ are from opposite active regions, we can estimate
\[ \lambda_l (1-2\epsilon) \le |\xi_3|=|\xi_1+\xi_2| \le 2 \operatorname{diam} (A_k^{\Circled{i}}) \le 4\sqrt{3} \epsilon \lambda_k. \]
Taking logarithms above
\[ l\le \log_2 \epsilon + k + 2+\log_2\frac{\sqrt{3}}{1-2\epsilon} \implies k > l -3-\log_2 \epsilon \]
for $\epsilon\le \epsilon_0$, as long as $\epsilon_0 < \frac{2 -\sqrt{3}}{4}$. Now \eqref{Pi^nonloc2} follows immediately and the definition of $N(\epsilon)$ is justified. \hfill $\square$

	\section{Proof of Theorem \ref{thm}}\label{sec:fin}
	\subsection{Exact estimates for $\piloc$}
	Recall \eqref{Pi^local2}; using \eqref{eq:fourier-flux} it reads
	\[
	\piloc =  \sum_{\xi^1 + \xi^2 + \xi^3 = 0} \widehat{u_{q}^{\Circled{2}}}  (\xi_2) \odot \widehat{u_{q}^{\Circled{3}}} (\xi_3) : \widehat{\nabla u_{q}^{\Circled{1}}} (\xi_1),
	\]
	where $\xi^1 \in \pm A_{q}^{\Circled{1}}, \xi^2 \in \pm A_{q}^{\Circled{2}}, \xi^3 \in \pm A_{q}^{\Circled{3}}$, since these are the active regions of respective $u_{q}^{\Circled{i}}$'s, see \eqref{uq-fourier}. By our construction, for $\xi^1 \in A_{q}^{\Circled{1}}$ and $\xi^2 \in A_{q}^{\Circled{2}}$, $\xi^1+\xi^2 \in A_{q}^{\Circled{3}}$ always holds, compare Figure \ref{fig:1} and recall from Section~\ref{ssec:full} that we used the blurs $A_{q}^{\Circled{3}}$ with twice the side length of the blurs $A_{q}^{\Circled{i}}$, $i=1,2$. Consequently, any $\xi^1 \in A_{q}^{\Circled{1}}$ and $\xi^2 \in A_{q}^{\Circled{2}}$ interacts via $\xi^1+\xi^2 +\xi^3=0$ with precisely one  $\xi^3 \in -A_{q}^{\Circled{3}}$; if only one of the signs of active regions of $\xi^1$, $\xi^2$ is flipped, there is no interaction; and any $\xi^1 \in -A_{q}^{\Circled{1}}$ and $\xi^2 \in -A_{q}^{\Circled{2}}$ interacts with precisely one  $\xi_3 \in A_{q}^{\Circled{3}}$. We thus arrive via \eqref{uq-fourier} and \eqref{uq-der-1} at
	\[ 
	\begin{aligned}
	&\piloc= \\ & - (\epsilon\lambda_{q})^{-6}
	2 \!\!\!\! \sum_{
	\substack{
	{\xi^1 \in A_{q}^{\Circled{1}}}\\ {\xi^2 \in A_{q}^{\Circled{2}}}}} \!\!\!\! \left( V_{\qm}^{\Circled{2}}+O(\epsilon) \right) \odot \left( V_{\qm}^{\Circled{3}}+O(\epsilon) \right) 
	: \left( V_{\qm}^{\Circled{1}} \otimes F_{\qm}^{\Circled{1}} +O(\epsilon) \right) \\
	&= - 2(\epsilon \lambda_{q})^{-6}|A_{q}^{\Circled{1}}| |A_{q}^{\Circled{2}}| \left( V_{\qm}^{\Circled{2}}  \odot  V_{\qm}^{\Circled{3}} : V_{\qm}^{\Circled{1}} \otimes F^{\Circled{1}}_{\qm} +O(\epsilon) \right).
	\end{aligned}
	\]	
	It holds $a \odot  b:c \otimes d =   (b \cdot c) \, (a \cdot d) + (a \cdot c) \, (b \cdot d)$. This, the fact that rotations, in particular $\rot$, do not alter the dot product, and the values of $V^{\Circled{i}}$ and $F^{\Circled{i}}$ provided in Section~\ref{ssec:planar} yield
	\[
	\piloc
	=  2(\epsilon \lambda_{q})^{-6}|A_{q}^{\Circled{1}}| |A_{q}^{\Circled{2}}| \left( 1+O(\epsilon) \right) \in 2(\epsilon \lambda_{q})^{-6}  [(\epsilon\lambda_{q}-1)^3, (\epsilon\lambda_{q}+1)^3] (1+ O(\epsilon)), 
	\]
	with the latter identity using \eqref{blur-size}.
		
		Summarizing, we obtain, for $O(\epsilon)$ independent of $q$
	\begin{equation}\label{piloc-value}
	2\left(\frac{\epsilon\lambda_{q}-1}{\epsilon\lambda_{q}}\right)^6 (1-O(\epsilon)) \le \piloc \le 2\left(\frac{\epsilon\lambda_{q}+1}{\epsilon\lambda_{q}}\right)^6 (1+O(\epsilon)).
	\end{equation}
	For any $\epsilon>0$ fixed, the lower and upper bound in \eqref{piloc-value} becomes, respectively, $2 \pm O(\epsilon)$ for large $q$'s. 
	
		\subsection{Estimate for $\pinonloc$} 
	Recall \eqref{Pi^nonloc2}. It reads via \eqref{eq:fourier-flux}
	\begin{align*}
	&\pinonloc\\
	&= \sum_{i=1,2,3} \sum_{l\le q}\sum_{k>\max \{q, l+N_\epsilon\}} 
	\sum_{\substack{\xi_1 + \xi_2 + \xi_3 =0 \\ |\xi^1| \in A_{k}^{\Circled{i}}, \; 
			|\xi^3| \in A_{l}^{\Circled{1}} } }\left(\widehat{u_{k}^{\Circled{i}}} (\xi_1) \otimes \widehat{u_{k}^{\Circled{i}}} (\xi_2) : \widehat{\nabla u_{l}^{\Circled{1}}} (\xi_3)\right).
    \end{align*}
	A brutal estimate for the innermost sum, using \eqref{uq-fourier} and \eqref{uq-der-1},  yields
	\[
	\begin{aligned}
 \sum_{ |\xi^1| \in A_{k}^{\Circled{i}}, \; |\xi^3| \in A_{l}^{\Circled{1}} }\left|\widehat{u_{k}^{\Circled{i}}} (\xi_1) \otimes \widehat{u_{k}^{\Circled{i}}} (\xi_2) : \widehat{\nabla u_{l}^{\Circled{1}}} (\xi_3)\right| &\le C \epsilon^{-6} \lambda_{k}^{-\frac{14}{3}} \lambda_{l}^{-\frac{4}{3}} |A_{k}^{\Circled{i}}| |A_{l}^{\Circled{1}}| \\
 &\le C \epsilon^{-3} \lambda_{k}^{-\frac{5}{3}} \lambda_{l}^{-\frac{4}{3}} (\epsilon\lambda_{l}+1)^3,
 \end{aligned}
	\]
	with the latter inequality using \eqref{blur-size} and that, since we interested in the limit $q \to \infty$, for any fixed $\epsilon$, $\epsilon \lambda_{k} \ge 1$ for $k\ge q$. 
 Therefore
	\[
	\begin{aligned}
	|\pinonloc| &\le C \sum_{l\le q}\sum_{k>\max \{q, l+N_\epsilon\}} \lambda_{k}^{-\frac{5}{3}} \max ({ \lambda_{l}^{\frac{5}{3}}, \epsilon^{-3} \lambda_{l}^{-\frac{4}{3}}}) \\
	&= C \sum_{l\le q} \left( \max \{1, \frac{1}{\epsilon \lambda_{l}} \} \right)^{3} \sum_{k>\max \{q, l+N_\epsilon\}} \left( \frac{\lambda_l}{\lambda_{k}} \right)^{\frac{5}{3}} .
	\end{aligned}
	\]
	Observe that the $\max$ in the outer sum is 1 if $l\ge l_\epsilon \coloneqq -\frac{\log\epsilon}{\log2}$. Therefore the sums split into
	\begin{align*}
	    &|\pinonloc|\\
	    &\le \sum_{l<l\epsilon} (\epsilon\lambda_l)^{-3} \sum_{k\ge q} \left( \frac{\lambda_l}{\lambda_k} \right)^{\frac53} + \sum_{l_\epsilon\le l \le q-N_\epsilon} \sum_{k\ge q} \left( \frac{\lambda_l}{\lambda_k} \right)^{\frac53} + \sum_{q-N_\epsilon<l\ge q} \sum_{k\ge l+N_\epsilon} \left( \frac{\lambda_l}{\lambda_k} \right)^{\frac53}
	\end{align*}
	(where the second sum might be empty). 
	The first summand above vanishes as $q \to \infty$ (and $\epsilon$ fixed), the second term is bounded independent of $q$ by $2^{-\frac53 N_\epsilon} =O(\epsilon)$, and the third sum is bounded by $N_\epsilon 2^{-\frac53 N_\epsilon} =O(\epsilon \log\epsilon)$, in view of $N_\epsilon = \lfloor 3- \log(\epsilon) \rfloor$. 
	Thus we have, for $q$ sufficiently large with respect to $\epsilon$,
	\begin{equation}\label{eq:pinonloc-est}
	|\pinonloc| \le O(\epsilon \log\epsilon).
	\end{equation}
	
	\subsection{Conclusion of the proof}
	Recall that the energy flux of $\Pi_q(U)$ is cubic in $U$, so by multiplying the original construction with $(c/2)^{1/3}$ we obtain a field such that by \eqref{piloc-value} and \eqref{eq:pinonloc-est}
	\[ c-O(\epsilon\log\epsilon) < \Pi_q(U) < c+O(\epsilon\log\epsilon) \]
	holds for sufficiently large $q$. Choosing $\epsilon$ in the construction sufficiently small with respect to $\delta$ then proves \eqref{limsup-liminf}.

	\bibliography{FluxBib}
	\bibliographystyle{amsplain}
	
\end{document}